\documentclass[12pt]{article}

\pdfoutput=1

\usepackage{amsthm,enumerate}
\usepackage{mathtools}
\usepackage[normalem]{ulem}
\usepackage[T1]{fontenc}
\usepackage[utf8]{inputenc}
\usepackage[thinc]{esdiff}
\usepackage{color}
\usepackage{fancyhdr}
\usepackage{calc}
\usepackage{url}
\usepackage{xcolor}
\usepackage{xspace}

\usepackage[all,cmtip]{xy} 
\usepackage{tikz-cd}

\usepackage[top=80pt,bottom=85pt,left=85pt,right=85pt]{geometry}
\usepackage{amsmath,amssymb,graphicx,float,color,tikz,cite,savesym,wasysym,subfigure,wrapfig}
\usepackage{caption}
\usepackage[debug,pageanchor=false]{hyperref}
\definecolor{link}{rgb}{1,0.45,0.05}
\hypersetup{colorlinks=true,linkcolor=link,citecolor=link,urlcolor=link,linktocpage}

\newcommand{\R}{\mathbb{R}}

\newcommand{\dd}{\mathrm{d}}
\newcommand{\ee}{\mathrm{e}}

\newcommand{\N}{\mathbb{N}}

\newcommand{\supp}{\text{\rm supp}}

\newcommand{\Lip}{\mathrm{Lip}}
\newcommand{\lip}{\mathrm{lip}}
\newcommand{\diam}{{\rm{diam\,}}}

\newcommand{\RCD}{\mathsf{RCD}}
\newcommand{\CD}{\mathsf{CD}}

\newcommand{\Ric}{{\rm Ric}}

\newcommand{\mm}{\mathfrak m}
\newcommand{\di}{\mathsf{d}}

\newcommand{\sfd}{\mathsf d}

\newcommand{\Ch}{\mathsf{Ch}}

\theoremstyle{plain}
\newtheorem{lemma}{Lemma}[section]
\newtheorem{theorem}[lemma]{Theorem}

\newtheorem{proposition}[lemma]{Proposition}

\newtheorem*{theorem*}{Theorem}
\newtheorem*{maintheorem*}{Main Theorem}

\theoremstyle{definition}

\newtheorem{definition}[lemma]{Definition}
\newtheorem*{definition*}{Definition}
\newtheorem*{remark*}{Remark}
\newtheorem{remark}[lemma]{Remark}

\makeatletter
\@addtoreset{equation}{section}
\makeatother

\begin{document}

	\begin{titlepage}

	\begin{center}

	\vskip .5in 
	\noindent

	{\Large \bf{Cheng's eigenvalue comparison on metric measure spaces and applications}}

	\bigskip\medskip
	 G. Bruno De Luca,$^1$  Nicol\`o De Ponti,$^2$
	Andrea Mondino,$^3$	Alessandro Tomasiello$^{4,5}$\\

	\bigskip\medskip
	{\small
$^1$ Leinweber Institute for Theoretical Physics at Stanford,\\
382 Via Pueblo Mall, Stanford, CA 94305, United States;
\\	
	\vspace{.3cm}
	$^2$
	Dipartimento di Matematica, Politecnico di Milano,\\
 Piazza Leonardo da Vinci 32, 20133 Milano, Italy;
\\	
	\vspace{.3cm}
	$^3$ Mathematical Institute, University of Oxford, Andrew-Wiles Building,\\ Woodstock Road, Oxford, OX2 6GG, UK;
\\
	\vspace{.3cm}
	$^4$
	Dipartimento di Matematica e Applicazioni, Universit\`a di Milano--Bicocca, \\ Via Cozzi 55, 20126 Milano, Italy;
\\	
	\vspace{.3cm}
$^5$ INFN, sezione di Milano--Bicocca.
		}

   \vskip .5cm 
	{\small \tt gbdeluca@stanford.edu, nicolo.deponti@polimi.it,\\ andrea.mondino@maths.ox.ac.uk, alessandro.tomasiello@unimib.it}
	\vskip .9cm 
	     	{\bf Abstract }
	\vskip .1in
	\end{center}

	\noindent
	Using the localization technique, we prove a sharp upper bound on the first Dirichlet eigenvalue of metric balls in  essentially non-branching $\mathsf{CD}^{\star}(K,N)$ spaces. This extends a celebrated result of Cheng to the non-smooth setting of metric measure spaces satisfying Ricci curvature lower bounds in a synthetic sense, via optimal transport. Rigidity and stability statements are provided for $\mathsf{RCD}^{\star}(K,N)$ spaces; the stability seems to be new even for smooth Riemannian manifolds. We then present some mathematical and physical applications: in the former, we obtain an upper bound on the $j^{th}$ Neumann eigenvalue in essentially non-branching $\mathsf{CD}^{\star}(K,N)$ spaces and a bound on the essential spectrum in non-compact  $\mathsf{RCD}^{\star}(K,N)$ spaces;  in the latter, the eigenvalue bounds correspond to general upper bounds on the masses of the spin-2 Kaluza-Klein excitations around general warped compactifications of higher-dimensional theories of gravity.
	\noindent


	\end{titlepage}

\tableofcontents

\section{Introduction} 
\label{sec:intro}

Given a domain $\Omega$ inside a space $X$, a classical question in spectral analysis is to determine meaningful bounds on $\lambda_1^D(\Omega)$, i.e. the first Laplacian Dirichlet eigenvalue of $\Omega$. This quantity is indeed connected to many other interesting mathematical objects, including the boundary hitting time of Brownian motion, and can be physically described as the principal frequency of a membrane fixed along its boundary. 
In general, the ambient space $X$ can possibly be curved and many meaningful bounds have been obtained in the literature in terms of quantities relative to the geometry of the domain and the space, see \cite{LMP23} for a recent monograph on this subject. These results are even more relevant when the domain has a peculiar geometric shape, and a typical example in this direction is given by a celebrated result of Cheng \cite{cheng-bound} that ensures a sharp upper bound of the first Dirichlet eigenvalue of a geodesic ball inside a $N$-dimensional Riemannian manifold with Ricci curvature bounded below by $K\in \R$. The result of Cheng enters the framework of comparison theorems in Riemannian geometry, indeed the bound is given in terms of the corresponding quantity in the model space, namely we have that $\lambda_1^D(B(x_0,r_0))$ is bounded from above by the first Dirichlet eigenvalue of the ball of radius $r_0$ in the model space of constant sectional curvature equal to $K/(N-1)$. Here and below, $B(x_0,r_0)$ denotes the ball centred at $x_0\in X$ with radius $r_0>0$.

In recent years, a tremendous amount of effort has been devoted to giving a synthetic description of lower Ricci curvature bounds and then to generalizing many classical results of Riemannian geometry to the non-smooth setting, starting from the seminal works of Sturm \cite{St,St1} and Lott–Villani\cite{LottVillani-Annals2009}. The resulting spaces, known as $\mathsf{CD}(K,N)$ spaces where the numbers $K\in \mathbb{R}$ and $N\in (1,\infty)$ stand for a synthetic lower bound on the Ricci curvature and an upper bound on the dimension, respectively, are now quite well understood. The primary aim of this article is to further our knowledge of this class of spaces by giving a generalization of Cheng's upper bound in this setting. More precisely, we prove the following theorem:

\begin{theorem}\label{th: Cheng dir}
Let $(X,\di,\mm)$ be an essentially non-branching $\CD^{\star}(K,N)$ space, $K\in \R$ and $N\in (1,\infty)$. Let $x_0\in X$ and $r_0>0$. Then 
\begin{equation}\label{eq:ChengBoundIntro}
\lambda_1^D(B(x_0,r_0))\le \lambda_{K,N,r_0},
\end{equation}
where $\lambda_{K,N,r_0}$ is defined in \eqref{def: firsteigen opt}.
\end{theorem} 
We also obtain the following rigidity result: if equality holds in \eqref{eq:ChengBoundIntro} and in addition $(X,\di,\mm)$ is infinitesimally hilbertian, i.e.\;$(X,\di,\mm)$ is an $\RCD^{\star}(K,N)$ space, then the geometry of $X$ is rigid (see Theorem \ref{th: Cheng dir-Rigid} for the precise statement). Combining the aforementioned rigidity with Gromov's precompactness, the stability of the $\RCD^*(K,N)$ condition, and the upper semi-continuity of the first Dirichlet eigenvalue on metric balls under pmGH convergence (see \cite[Lemma 2.10]{AH18}), we then prove a stability result: if equality is almost attained in \eqref{eq:ChengBoundIntro}, then the space is pmGH-close to a $(K,N)$-cone (see Theorem \ref{th: Cheng dir-Stab} for the precise statement). Let us stress that such a stability statement for Cheng's eigenvalue comparison seems new even for smooth Riemannian manifolds.
\medskip

The essentially non-branching condition, originally introduced in \cite{rajala-sturm}, is a technical hypothesis that, roughly speaking, prevents geodesics in a space from branching excessively in a measure-theoretic sense. It is well known that this assumption plays a crucial role in the application of the so-called \emph{localization technique} \cite{cavalletti-mondino-inv}, the main technical tool used to prove Theorem \ref{th: Cheng dir}. Fundamentally, this technique reduces what is initially a high-dimensional problem on $X$ to a family of simpler one-dimensional problems. The corresponding one-dimensional analysis is carried out in Proposition \ref{prop: modeleigen} and Theorem \ref{th: 1dim}.

We emphasize that Theorem \ref{th: Cheng dir} is satisfied by a broad class of spaces, including reversible Finsler manifolds and weighted Riemannian manifolds, as well as potentially nonsmooth spaces such as Alexandrov spaces, Ricci limit spaces and, more generally, $\RCD(K,N)$ spaces. Moreover, it has the advantage of providing a sharp upper bound, with the quantity $\lambda_{K,N,r_0}$ that represents the first eigenvalue of the corresponding one-dimensional model space described in Definition \ref{def: model space}. Although its exact value is not explicit, it can nonetheless be effectively bounded from above, as we discuss in Remark \ref{rem: expl value} for the full range $N\in (1,\infty)$.

A couple of direct consequences of Theorem \ref{th: Cheng dir} are discussed in Section \ref{sec: consequence}. Assuming that the space has finite diameter, Theorem \ref{th: Cheng neum} provides an upper bound on the $j^{th}$ Laplacian eigenvalue of $(X,\di,\mm)$. For non-compact spaces, we instead establish an upper bound for the bottom of the essential spectrum of the Laplacian; see Theorem \ref{th: essspect}.

While Cheng's inequality has already been studied in the context of smooth weighted Riemannian manifolds \cite{setti-eigenvalues} (see also \cite{YinHe2015} for the Finslerian counterpart), the novelty of our work lies in extending the result to spaces that may contain singularities. This is particularly significant in the study of gravity compactifications, where it has been observed that a large class of spaces considered in this field belongs to the class of $\RCD(K,N)$ spaces \cite{deluca-deponti-mondino-t}. In this context, bounds on Laplacian eigenvalues are especially relevant, as they directly correspond to bounds on the masses of spin-2 Kaluza–Klein particles on general warped compactifications. Specifically, denoting by $\Lambda$ and $f$ respectively the cosmological constant and the warp function, in Section \ref{sec: phys}, we apply Thm.~\ref{th: Cheng dir} and previous curvature bounds \cite{deluca-t-leaps, deluca-deponti-mondino-t-entropy} to derive
\begin{proposition}\label{prop:kk-intro}
	The mass of the $j$-th spin-two Kaluza Klein fluctuation around a general $d$-dimensional vacuum compactification \eqref{eq:metAns} of a $D$-dimensional gravitational theory that satisfies the REC is bounded as
	\begin{equation}
		m^2_j \leqslant \lambda_{K, N, r_0}
	\end{equation}
	with
	\begin{equation}
		K = \Lambda-\frac{N + d-2}{(D-2)(N-D+d)}\sigma^2\;, \qquad \sigma \equiv \sup{|\nabla f|}\;,\qquad r_0 = \frac{\diam}{2j}\;,
	\end{equation}
	$\forall N > D-d$, where $\mathrm{diam}$ is the diameter of the internal space.
\end{proposition}

We conclude this introduction by noting that, to the best of our knowledge, the only prior version of Cheng’s inequality applicable to possibly non-smooth metric measure spaces was obtained in \cite{kristaly}. There, the author derives a Cheng-type bound assuming a Bishop–Gromov-type volume monotonicity with \emph{negative} curvature parameter and an additional local density hypothesis on the volume measure.

\medskip
\paragraph{\em\bfseries Acknowledgments} GBDL's work is supported in part
by the NSF Grant PHY-231042. NDP is supported by the INdAM-GNAMPA Project ``Proprietà qualitative e regolarizzanti di equazioni ellittiche e paraboliche'', codice CUP $\#E5324001950001\#$. 
AM acknowledges support from the European Research Council (ERC) under the European Union's Horizon 2020 research and innovation programme, grant agreement No.~802689 ``CURVATURE''. AT is supported in part by the INFN, and by the MUR-PRIN contract 2022YZ5BA2.

The authors wish to thank Shouhei Honda for pointing out the useful reference \cite[Lemma 2.10]{AH18} and for a helpful discussion on the proof of Theorem \ref{th: Cheng dir-Stab}.

\section{Preliminaries}\label{sec:prel}
\subsection{Metric measure spaces}

We start by recalling some basic concepts of analysis in metric measure spaces. We refer to \cite{AGS2} as a general reference on the subject.

By \emph{metric measure space} we mean a triple $(X,\di,\mm)$ where $(X,\di)$ is a complete and separable metric space endowed with a non-negative Borel measure $\mm$ finite on bounded sets and such that $\supp(\mm)=X$. The open metric ball in $(X,\di)$ with center $x_0$ and radius $r$ will be denoted by $B(x_0,r)$, i.e. $B(x_0,r):=\{x\in X\,:\; \di(x,x_0)<r\}$. We use the notation $C(Y,Z)$ for the set of continuous functions from $Y$ to $Z$, $Y,Z$ topological spaces; when $Z=\mathbb{R}$, we will drop the letter in the notation and simply write $C(Y)$. For instance, $C([0,1],X)$ will be the set of continuous curves with values in the metric space $(X,\di)$.

For every $\Omega\subset X$ open, we denote by $\Lip(\Omega)$,
$\Lip_{\textsf{bs}}(\Omega)$ and $\Lip_{\textsf{c}}(\Omega)$ the set of Lipschitz functions, Lipschitz functions with bounded support in $\Omega$ and Lipschitz functions with compact support in $\Omega$, respectively.

The slope of a function $f\in \Lip_{\textsf{bs}}(X)$ at a point $x\in X$ is defined as
$$\lip(f)(x):=\limsup_{y\to x}\frac{|f(y)-f(x)|}{\di(y,x)},$$
with the convention $\lip(f)(x)=0$ if $x$ is an isolated point in $(X,\di)$.

Given a function $f\in L^2(X,\mm)$,  the Cheeger energy $\Ch:L^2(X,\mm)\to[0,+\infty]$ is defined as
$$\Ch(f):=\inf\left\{\liminf_{n\to\infty} \frac{1}{2}\int_X \lip(f_n)^2\,\dd \mm : (f_n)\in \Lip_{bs}(X), f_n\to f\in L^2(X,\mm)\right\}.$$
The domain of the Cheeger energy is the set $D(\Ch):=\{f\in L^2(X,\mm), \ \Ch(f)<+\infty\}$. The global Sobolev space $W^{1,2}(X,\di,\mm)$ is the Banach space
$$W^{1,2}(X,\di,\mm):=D(\Ch), \quad \|f\|^2_{W^{1,2}}:=\|f\|^2_{L^2}+2\Ch(f).$$

For every $f\in D(\Ch)$ the Cheeger energy can be represented in terms of the so-called minimal weak upper gradient $|Df|$ as $\Ch(f)=\frac{1}{2}\int_X |Df|^2$. Given $f,g\in W^{1,2}(X,\di,\mm)$, we recall some well-known properties of the minimal weak upper gradient: it is a \emph{local} object in the sense that $|Df|=|Dg|$ $\mm$-a.e. in $\{f=g\}$; the \emph{chain rule} holds, i.e. for every $\varphi\in \Lip(\R)$ with $\varphi(0)=0$ we have $\varphi(f)\in W^{1,2}(X,\di,\mm)$ and $|D(\varphi(f))|=|\varphi'(f)||Df|$; the \emph{Leibniz rule} is satisfied, meaning that for every $\eta\in \Lip(X)\cap L^{\infty}(X,\mm)$ we have $\eta f\in W^{1,2}(X,\di,\mm)$ and $|D(f\eta)|\le |Df||\eta|+|D\eta||f|$.

Let $\Omega\subset X$ be open. Following \cite{AH18}, we define the local \emph{Dirichlet} Sobolev space $W^{1,2}_0(\Omega)$ as the $W^{1,2}$-closure of $\Lip_{\textsf{c}}(\Omega)$ and we endow it with the norm 
\[\|f\|^2_{W^{1,2}_0(\Omega)}:=\|f\|^2_{L^2(\Omega,\mm)}+\||Df|\|^2_{L^2(\Omega,\mm)}.\]
We also set
\begin{equation}\label{eq: def lambda_1^D}
    \lambda_1^D(\Omega):=\inf\left\{\frac{\int_{\Omega}|Df|^2\,\dd \mm}{\int_{\Omega}|f|^2\,\dd \mm}\, : \, f\in W^{1,2}_0(\Omega), \; f\not\equiv 0\right\},
\end{equation}
which corresponds to the first (variational) Dirichlet eigenvalue of the set $\Omega$.

Similarly, we can also define the Neumann (variational) eigenvalues of the space $(X,\di,\mm)$. In order to do so, it is convenient to firstly introduce the Rayleigh quotient of a non-null function $f\in W^{1,2}(X,\di,\mm)$, namely
$$\mathcal{R}(f):=\frac{2\Ch(f)}{\int_X |f|^2\,\dd \mm}.$$

The $j^{th}$ Neumann eigenvalue is then defined as
\begin{equation}\label{def: m-th Neum eigen}
\lambda^N_j:= \min_{V_{j+1}}\, \max_{f\in V_{j+1},\\ f\not\equiv 0} \ \mathcal{R}(f),\qquad j\in \N=\{0,1,2,\dots\},
\end{equation}
where $V_{j}$ denotes an arbitrary $j$-dimensional subspace of $W^{1,2}(X,\di,\mm)$. Notice that $\lambda_0^N=0$ whenever $\mm(X)<\infty.$

The space of probability measures over $X$ is denoted by $\mathcal{P}(X)$, while $\mathcal{P}_2(X)\subset\mathcal{P}(X)$ is the subset of the probability measures with finite second moment. We endow $\mathcal{P}_2(X)$ with the $2$-Kantorovich-Wasserstein distance $W_2$ given by the optimal transport problem with cost $\di^2$, making it a complete and separable metric space (e.g., see \cite{Vil}). The space $(\mathcal{P}_2(X),W_2)$ is geodesic if the underlying space $(X,\di)$ is geodesic. Recall that by definition a metric space $(X,\di)$ is geodesic if for every $x,y\in X$ there exists $\gamma\in \mathrm{Geo}(X)$ with $\gamma(0)=x$ and $\gamma(1)=y$, where we set
$$ \mathrm{Geo}(X):=\{\gamma\in C([0,1],X)\, : \, \di(\gamma(s),\gamma(t))=|t-s|\di(\gamma(0),\gamma(1)), \quad \forall s,t\in[0,1]\}.$$

Any geodesic $(\mu_t)_{t\in [0,1]}$ in $(\mathcal{P}_2(X),W_2)$ can be lifted to a measure $\nu\in\mathcal{P}(\mathrm{Geo}(X))$
such that $(e_t)_{\sharp}\nu=\mu_t$ for every $t\in [0,1]$, where $e_t:\mathrm{Geo}(X)\rightarrow X$ is the evaluation map defined as $e_t(\gamma):=\gamma(t)$. We denote by $\mathrm{OptGeo}(\mu_0,\mu_1)$ the space of all optimal dynamical plans $\nu$, i.e. the measures $\nu\in\mathcal{P}(\mathrm{Geo}(X))$ such that $(e_0,e_1)_{\sharp}\nu$ is an optimal transference plan between $\mu_0$ and $\mu_1$ for the optimal transport problem in $X$ with cost given by $\di^2$.

A subset $G\subset \mathrm{Geo}(X)$ is called non-branching if, given $\gamma_1,\gamma_2\in G$ such that there exists $t\in (0,1)$ with $\gamma_1(s)=\gamma_2(s)$ for every $s\in [0,t]$, we have $\gamma_1(s)=\gamma_2(s)$ for every $s\in [0,1]$. The metric space $(X,\di)$ is called non-branching if $\mathrm{Geo}(X)$ is non-branching. An important definition for us is the following \cite{rajala-sturm}: $(X, \di, \mm)$ is called \emph{essentially non-branching} if for all $\mu_0,\mu_1 \in\mathcal{P}_2(X)$ any $\nu\in \mathrm{OptGeo}(\mu_0,\mu_1)$ is concentrated on a Borel non-branching set $G\subset \mathrm{Geo}(X)$.

\subsection{Curvature dimension condition and Laplacian}

Let $\kappa\in \R$ and $\theta\geqslant 0$,  set
\begin{equation}\label{eq: sine func curv}
\mathfrak{s}_\kappa(\theta):=\begin{cases}
\frac{1}{\sqrt{\kappa}}\sin(\sqrt{\kappa}\theta) \quad &\textrm{if} \ \kappa>0\,,\\
\theta \quad &\textrm{if} \ \kappa=0\,,\\
\frac{1}{\sqrt{-\kappa}}\sinh(\sqrt{-\kappa}\theta) \quad &\textrm{if} \ \kappa<0\,.
\end{cases}
\end{equation}
For $t\in[0,1]$, consider the quantity
\begin{equation}\label{eq: sigma func curv}
\sigma_{\kappa}^{(t)}(0):=t\,,\hspace{2.5cm} \sigma_{\kappa}^{(t)}(\theta):=\frac{\mathfrak{s}_\kappa(t\theta)}{\mathfrak{s}_\kappa(\theta)} \,,
\end{equation}
with $\theta>0$ if $\kappa\leqslant 0$ and $\theta\in (0,\pi/\sqrt{\kappa})$ if $\kappa>0$. 
\\Given $K\in \mathbb{R}$ and $N\in (1,\infty)$, we put $\sigma_{K,N}^{(t)}:=\sigma_{K/N}^{(t)}$. Finally, for $t\in [0,1]$, consider the functions
\begin{align*}\label{eq: tau func curv}
&\tau_{K,N}^{(t)}(\theta):=\begin{cases}t^{1/N}\sigma_{K,N-1}^{(t)}(\theta)^{(N-1)/N} \qquad &\textrm{if} \ \theta\in (0,D_{K,N})\,,\\
+\infty &\textrm{otherwise},
\end{cases} \quad \textrm{for}\ t>0\,,\\
&\tau_{K,N}^{(0)}(\theta)=0.
\end{align*}

Let $(X,\sfd, \mm)$ be a metric measure space and fix $\alpha>1$. Given $\mu\in \mathcal{P}(X)$,  the $\alpha$-Tsallis entropy of $\mu$ with respect to $\mm$ is defined by
\begin{equation}
\mathcal{S}_\alpha(\mu):=
\begin{cases}
\frac{1}{\alpha-1}\left(1-\int_X \rho^{\alpha}\dd\mm\right), \qquad &\text{if} \ \mu=\rho\mm\ll \mm \text{ and } \rho^\alpha\in L^1(X,\mm)\,,  \\
+\infty \qquad &\textrm{otherwise.}
\end{cases}
\end{equation}

We are now ready to recall the definition of the curvature dimension condition  \cite{St,St1,LottVillani-Annals2009},  its reduced version \cite{BS-JFA}, and their Riemannian refinements \cite{AGS1}, \cite{Gigli}, \cite{AGMR}.
\begin{definition}\label{def: CD(K,N)}
\textbf{$\CD$ condition}: A metric measure space $(X,\di,\mm)$ satisfies the curvature-dimension condition $\CD(K,N)$, $K\in \R$ and $N\in (1,\infty)$, if for any pair of absolutely continuous measures $\mu_0=\rho_0\mm$, $\mu_1=\rho_1\mm\in \mathcal{P}_2(X)$  there exists  $\eta\in \mathcal{P}(\mathrm{Geo}(X))$ such that $\mu_t:=(\ee_t)_{\sharp}\eta=\rho_t\mm$ is a $2$-Wasserstein geodesic between $\mu_0$ and $\mu_1$ and for all $t\in[0,1]$ and all $N'\ge N$ it holds
\begin{equation}\label{eq: def CD(K,N)}
{\mathcal S}_{\frac{N'-1}{N'}}(\mu_t)\geq -N' +N' \int_{X\times X}\left[\tau_{K,N'}^{(1-t)}(\di(x,y))\rho_0(x)^{-\frac{1}{N'}}+\tau_{K,N'}^{(t)}(\di(x,y))\rho_1(y)^{-\frac{1}{N'}}\right]\dd\pi(x,y)
\end{equation}
where $\pi=(e_0,e_1)_{\sharp}\eta$ is an optimal plan between $\mu_0$ and $\mu_1$ induced by $\eta$.
\smallskip

\textbf{Reduced condition}: $(X,\di,\mm)$ satisfies the \emph{reduced} curvature-dimension condition $\CD^*(K,N)$, $K\in \R$ and $N\in (1,\infty)$, if \eqref{eq: def CD(K,N)} holds for any pair of $\mm$-absolutely continuous measures with bounded support and with $\tau_{K,N'}$ replaced by $\sigma_{K,N'}$.

\textbf{$\RCD$ condition}: $(X,\di,\mm)$ is an $\RCD(K,N)$ space (resp.\;$\RCD^*(K,N)$), $K\in \R$ and $N\in (1,\infty)$, if it is a $\CD(K,N)$ (resp.\;$\CD^*(K,N)$) space and the Sobolev space $W^{1,2}(X,\di,\mm)$ is a Hilbert space.
\end{definition}

From the definitions it follows directly that $\CD(K,N)$ implies $\CD^*(K,N)$.  A deep result proved in \cite{CaMi} (see also \cite{Li-AMPA}) is that, actually,  $\RCD(K,N)$ and $\RCD^*(K,N)$ are equivalent.

Thanks to the Hilbertianity of $W^{1,2}(X,\di,\mm)$, it is well known that on an $\RCD(K,N)$ space one can take the subdifferential of the Cheeger energy to define an operator $-\Delta: D(\Delta)\subset L^2(X,\mm)\to L^2(X,\mm)$ called (minus) \emph{Laplacian}, which automatically satisfies the integration by parts formula
\[
\int_X -f\,\Delta f\; \dd\mm=\int_X|Df|^2\, \dd\mm \qquad \forall f\in D(\Delta).
\]
The domain $D(\Delta)$ is dense in $W^{1,2}(X,\di,\mm)$ and the Laplacian is a non-negative, densely defined, self-adjoint operator on $L^2(X,\mm)$, thus entering in the classical framework of spectral theory. In particular,  the \emph{essential spectrum} of $-\Delta$  consists of cluster points of the spectrum and eigenvalues having infinite multiplicity.

\subsection{One dimensional metric measure spaces}
One dimensional model spaces for the curvature condition will play an important role in the proofs of the main results. We start with the following.
\begin{definition}
Let $K\in \mathbb{R}, N\in (1,\infty)$. A function $h\ge 0$ defined on an interval $I\subset \R$ is a $\mathsf{CD}(K,N)$ density on $I$ if for every $\theta_0,\theta_1\in I$ and every $t\in [0,1]$ it holds
\begin{equation}\label{eq:defhCD}
 h(t\theta_0+(1-t)\theta_1)^{\frac{1}{N-1}}\ge \sigma_{K,N-1}^{(t)}(|\theta_1-\theta_0|)h(\theta_1)^{\frac{1}{N-1}}+\sigma_{K,N-1}^{(t)}(|\theta_1-\theta_0|)h(\theta_0)^{\frac{1}{N-1}}.
\end{equation}
\end{definition}

We recall that if $h$ is a $\mathsf{CD}(K,N)$ density on a closed interval $I$ then $(I, |\cdot|, h(\theta)dt)$ verifies the $\mathsf{CD}(K,N)$ condition. Conversely, if the m.m.s. $(\mathbb{R}, |\cdot| , \mu)$ verifies the $\mathsf{CD}(K,N)$ condition and $I = \mathsf{supp}(\mu)$ is not a point, then $\mu\ll\mathcal{L}^1$ and there exists an $\mathcal{L}^1$-a.e.\;representative  of the density $h=d\mu/d\mathcal{L}^1$ which is a $\mathsf{CD}(K,N)$ density on $I$. Here and in the sequel, by $\mathcal{L}^1$ we denote the $1$-dimensional Lebesgue measure. Moreover, for every $\mathsf{CD}(K,N)$ density $h$ on $I$ we have that \cite{CaMi}:
\begin{itemize}
    \item $h$ is bounded and lower semicontinuous on $I$, and locally Lipschitz continuous in its interior;
    \item $h$ is strictly positive in the interior of $I$, unless it vanishes identically;
    \item $h$ is locally semi-concave in the interior. 
\end{itemize}

\begin{definition}\label{def: model space}
    The one-dimensional model metric measure space $(I_{K,N},\di_{eu},\mm_{K,N})$ with parameters $K\in \mathbb{R}$ and $N\in(1,\infty)$ is given by
    \begin{equation}
      \begin{cases}
        I_{K,N}=\overline{[0,D_{K,N})},\\
        \di_{eu}(\theta_1,\theta_2)=|\theta_2-\theta_1|,\\
        \mm_{K,N}(d\theta)=h_{K,N}(\theta)\; \mathcal{L}^1(d\theta),
    \end{cases}
    \end{equation}
where
 \begin{equation}
 D_{K,N}:=
  \begin{cases} \pi \sqrt{\frac{N-1}{K}} \quad &\textrm{for } K> 0,\\
+\infty, &\textrm{for } K\leq 0,
   \end{cases} \qquad \textrm{and} \quad  h_{K,N}(\theta):=\mathfrak{s}^{N-1}_{K/(N-1)}(\theta).
    \end{equation}
\end{definition}

\begin{remark}[On the one-dimensional model spaces]\label{Rem:1-DimModel} Notice that when $N\geq 2$ is integer, $( I_{K,N},   \di_{eu},  \mm_{K,N})$  can be seen as one-dimensional projection of $\mathbb{M}^N_{K/(N-1)}$,  the simply connected $N$-dimensional space form of constant sectional curvature $K/(N-1)$ (i.e., round sphere for $K>0$, Euclidean space for $K=0$ and hyperbolic space for $K<0$), in the sense that the Riemannian volume form of $\mathbb{M}^N_{K/(N-1)}$  in geodesic polar coordinates  writes as $\mathrm{dvol}_{\mathbb{M}^N_{K/(N-1)}}=  \mm_{K,N}  \otimes \rm{dvol}_{\mathbb{S}^{N-1}}$. The interval $[0,r_0)\subset I_{K,N}$ corresponds to a metric ball in $\mathbb{M}^N_{K/(N-1)}$ of radius $r_0$ centred at the origin of the geodesic polar coordinates.
Since metric balls in $\mathbb{M}^N_{K/(N-1)}$ achieve equality in Cheng's inequality, one can see that the segments $[0,r_0)\subset I_{K,N}$ (where  $(I_{K,N},\di_{eu},\mm_{K,N})$ is as above)  are the optimal one-dimensional model spaces for the Cheng's inequality. In contrast,  for the isoperimetric problem, the class of model spaces is more rich and complicated \cite{Milm-JEMS}.
\end{remark}

Let $K\in \mathbb{R}$, $N\in(1,\infty)$, and $r_0>0$. We define the optimal Dirichlet eigenvalue of parameters $K,N,r_0$ as the quantity
\begin{equation}\label{def: firsteigen opt}
    \lambda_{K,N,r_0}:=\inf_{u}\frac{\int_{0}^{r_0}|u'|^2\,\dd  \mm_{K,N}}{\int_{0}^{r_0}|u|^2\,\dd  \mm_{K,N}},
\end{equation}
where the infimum runs over all the non-identically vanishing Lipschitz functions $u:[0,+\infty)\to \R$ with $u(\theta)=0$ for all $\theta\ge r_0$. It is easily seen that the function $r_0\mapsto \lambda_{K,N,r_0}$ is monotone decreasing. Moreover, when $N\geq 2$ is integer, $\lambda_{K,N,r_0}$ coincides with the first Dirichlet eigenvalue of a metric ball of radius $r_0$ in the simply connected $N$-dimensional space form $\mathbb{M}^N_{K/(N-1)}$  of constant sectional curvature $K/(N-1)$ (see Remark \ref{Rem:1-DimModel}).

We denote by $\varphi_{K,N,r_0}$ the first Dirichlet eigenfunction for the interval $[0,r_0)\subset I_{K,N}$, where $( I_{K,N},   \di_{eu},  \mm_{K,N})$ is the model 1-dimensional space described above. 
By eigenfunction (relative to an eigenvalue $\lambda$), we mean here a function $\varphi_1$ which is in the kernel of the linear operator
\begin{equation}
L: H\to H', \quad Lu= u''+ (\log h_{K,N})' u' + \lambda u\,,
\end{equation}
where $H'$ denotes the dual of $H$, and the latter is the closure of the space of Lipschitz functions $u:[0,+\infty)\rightarrow \R$ such that $u(\theta)=0$ for all $\theta\ge r_0$, with the closure taken with respect to the topology of the weighted Sobolev space $W^{1,2}(I_{K,N},\di_{eu},\mm_{K,N})$. In the next proposition we show the existence and some basic properties of $\varphi_{K,N,r_0}$.

\begin{proposition}\label{prop: modeleigen}
Let $(I_{K,N},\di_{eu},\mm_{K,N})$ be the one-dimensional model metric measure space with parameters $K\in \mathbb{R}$ and $N\in(1,\infty)$. Then, for every $r_0\in (0, D_{K,N})$:
\begin{enumerate}
\item There exists a unique eigenfunction $\varphi_{K,N,r_0}:[0,+\infty)\to \mathbb{R}$ relative to $\lambda_{K,N,r_0}$, up to multiplication by a non-zero constant. 
\item $\varphi_{K,N,r_0}$ is Lipschitz continuous on $[0,+\infty)$ and of class $C^\infty$ on $(0,r_0)$, it has constant sign and it satisfies
\begin{equation}\label{eq: eigen ode}
 \varphi''_{K,N,r_0}+(\log{h_{K,N}})' \varphi'_{K,N,r_0}=-\lambda_{K,N,r_0}\varphi_{K,N,r_0} \qquad \textrm{in} \ (0,r_0).
\end{equation}
\item Under the convention that $\varphi_{K,N,r_0}$ is non-negative, then 
\begin{equation}\label{eq: eigen-Decreas}
\varphi_{K,N,r_0}(\theta)>0 \quad \textrm{ and } \quad \varphi_{K,N,r_0}'(\theta)< 0 \textrm{ for all } \theta\in (0,r_0).
\end{equation}
Moreover, $\displaystyle\lim_{\theta\downarrow 0} \varphi'_{K,N,r_0}(\theta)=0$. 

\end{enumerate}
\end{proposition}

\begin{proof}
As recalled above, in case $N\geq 2$ is a natural number, then $[0,r_0)\subset I_{K,N}$ (where  $(I_{K,N},\di_{eu},\mm_{K,N})$ is as in Definition \ref{def: model space})  is the quotient of the metric ball $B(\bar{x}, r_0)$ in the model space of constant sectional curvature $K/(N-1)$ after taking polar coordinates centred at $\bar{x}$ and identifying all the points in the sphere at the same distance from $\bar{x}$. It is a classical fact that the first Dirichlet eigenfunction in a metric ball in a model space is radial, smooth, unique, positive (up to multiplication by a non-zero constant) and satifies \eqref{eq: eigen ode}. Thus the claim follows by looking at such an eigenfunction in the model space as a function of the distance from the center of the ball.

 In case of a general $N\in (1,\infty)$, let us recall the main steps of the proof in order to keep the note as self-contained as possible. Most of the arguments are classical, and can be derived immediately from general results in calculus of variations. Notice indeed that $\varphi_{K,N,r_0}$ corresponds to the minimizer of the following functional
 \[\inf_{\{u \ : \ u(r_0)=0\}}\int_0^{r_0}F(\theta,u,u')d\mathcal{L}^1(\theta) \quad \textrm{with }  F(\theta,z,p):=\big(p^2+\lambda(z^2-1)\big)\mathfrak{s}^{N-1}_{K/(N-1)}(\theta).\]
 To simplify the notation, during the proof we will use $\varphi_1$ in place of $\varphi_{K,N,r_0}$ and $\lambda$ in place of $\lambda_{K,N,r_0}$.

 \textbf{Step 1.} Existence of a minimizer.
\\Let $W^{1,2}(I_{K,N},\di_{eu},\mm_{K,N})$ be the Sobolev space of $L^2$ functions with $L^2$ weak first derivative in the interval $I_{K,N}$ endowed with the measure $\mm_{K,N}$. Define $H$ to be the closure (in $W^{1,2}(I_{K,N},\di_{eu},\mm_{K,N})$) of the space of Lipschitz functions $u:[0,+\infty)\rightarrow \R$ such that $u(\theta)=0$, for all  $\theta\geq r_0$. It is easy to see that $H$ is a Hilbert space. Consider the functional
\begin{equation}\label{eq:defJ}
J: H\setminus\{0\}\to \R, \quad J(u):= \frac{\int |u'|^2 \,  \dd \mm_{K,N}}{ \int |u|^2 \,  \dd \mm_{K,N} }.
\end{equation}
Notice that $J$ is homogenous of degree $0$, so it is completely determined by its values for those $u\in H$ with $\|u\|_{L^2(I_{K,N},\mm_{K,N})}=1$.
Using that $H$ compactly embeds in $L^2(I_{K,N},\mm_{K,N})$ and that $u\mapsto \int |u'|^2 \,  \dd \mm_{K,N}$ is lower semicontinuous with respect to weak convergence in $H$, it follows from the direct method in the calculus of variations that $J$ admits a (a priori non-unique) global minimizer $\varphi_1\in H$. Denote
\begin{equation}\label{eq:dellambda}
\lambda:=J(\varphi_1).
\end{equation}
It is easy to see that $\lambda>0$. We next prove that $\varphi_1$ is an eigenfunction relative to the eigenvalue $\lambda$.
\medskip

 \textbf{Step 2.}  $\varphi_1$ is an eigenfunction.
\\ Fix an arbitrary $\psi\in H$. Since $\varphi_1\not \equiv 0$ $\mm_{K,N}$-a.e., there exists some small $\varepsilon>0$ such that $\|\varphi_1+t \psi\|_{L^2(I_{K,N},\mm_{K,N})}>0$ for all $|t|<\varepsilon$. Consider the function
$J_{\varphi_1}(\cdot):(-\varepsilon, \varepsilon)\to \R$ defined as
\begin{equation}\label{eq:defJ1}
J_{\varphi_1}(t):= \frac{\int |(\varphi_1+t\psi)'|^2 \,  \dd \mm_{K,N}}{ \int |\varphi_1+t\psi|^2 \,  \dd \mm_{K,N} }.
\end{equation}
It is easy to see that $J_{\varphi_1}(\cdot)$ is a rational function of $t$ with a minimum at $t=0$. It follows that $J_{\varphi_1}'(0)=0$, that is
$$
\frac{\int 2 \varphi_1'  \psi'  \,  \dd \mm_{K,N}}{ \int |\varphi_1|^2 \,  \dd \mm_{K,N}} - \frac{\int |\varphi_1'|^2 \,  \dd \mm_{K,N} } { \left(\int |\varphi_1|^2 \,  \dd \mm_{K,N}\right)^2} \; \int 2 \varphi_1 \psi  \,  \dd \mm_{K,N}=0.
$$
Using that $J_{\varphi_1}(0)=J(\varphi_1)=\lambda$, from \eqref{eq:dellambda} and simplifying, we obtain
\begin{equation}\label{eq:varphi1EFweak}
\int \left(\varphi_1' \psi'- \lambda \varphi_1 \psi \right) \, \dd \mm_{K,N}=0.
\end{equation}
Since \eqref{eq:varphi1EFweak} holds for all $\psi\in H$, this means that $\varphi_1$ is in the kernel of the linear operator
\begin{equation}
L: H\to H', \quad Lu= u''+ (\log h_{K,N})' u' + \lambda u,
\end{equation}
i.e. $\varphi_1$ is an eigenfunction relative to the eigenvalue $\lambda$.
\medskip

 \textbf{Step 3}. Regularity of  $\varphi_1$.
 \\From Step 2 we know that $\varphi_1$ is a weak solution of the differential equation
 \begin{equation}\label{eq:varphi1diffEq}
 \varphi_1''+ (\log h_{K,N})' \varphi_1' + \lambda \varphi_1=0.
 \end{equation}
 From the explicit expression of $\mm_{K,N}$,  it is clear that $h_{K,N}>0$ and is smooth on $(0,r_0)$. One can then bootstrap \eqref{eq:varphi1diffEq} to get that $\varphi_1$ is of class $C^\infty((0,r_0))$ and it is a classical solution of the differential equation \eqref{eq:varphi1diffEq} in the interval $(0,r_0)$. Using that  $h_{K,N}$ is continuous and positive on $(0,r_0]$ (here we use that $r_0<D_{K,N}$ in case $K>0$), it is easy to see that  $\varphi_1(r_0)=0$, and that $\varphi_1$ is continuous in $[0,+\infty)$ and Lipschitz on $[\varepsilon, +\infty)$ for every $\varepsilon>0$.
 \medskip

  \textbf{Step 4}. Let $w\in {\rm Ker}(L)$ and let $w_+:=\max\{w,0\}$. Then also $w_+\in  {\rm Ker}(L)$.
 \\If $w_+=0$ $\mm_{K,N}$-a.e. then clearly $w_+\in   {\rm Ker}(L)$. We can then assume that $w_+\not \equiv 0$. It is easy to see that $w_+\in H$. Using $w_+$ as a test function in the weak formulation of $L w_+=0$ (i.e. \eqref{eq:varphi1EFweak} with $\varphi_1$ replaced by $w_+$) gives
$$
\int \left[(w_+')^2- \lambda  w_+^2 \right] \, \dd \mm_{K,N}=0.
$$
This means that $J(w_+)=\lambda=J(\varphi_1)$, that is $w_+$ is a global minimizer for $J$. Repeating the argument of step 2 (with $w_+$ in place of $\varphi_1$), we obtain that $w_+\in  {\rm Ker}(L)$.
\medskip

  \textbf{Step 5}.   All functions in ${\rm Ker}(L)\setminus\{0\}$ do not change sign.
  \\Let $w\in {\rm Ker}(L)\setminus\{0\}$. Repeating the arguments in step 3 replacing $\varphi_1$ by $w$, we obtain that $w\in C([0,r_0])\cap C^\infty((0,r_0))$. Replacing $w$ by $-w$ if necessary, we may assume that $w>0$ on some open interval $I_0$. By step 4, we know that $w_-:=\max\{-w,0\}$ is an element of ${\rm Ker}(L)$. By construction, $w_-\equiv 0$ on $I_0$. By the uniqueness of solutions to ODEs (Cauchy-Lipschitz theorem), it follows that $w_- \equiv 0$ on $[0, r_0]$ and thus $w(x)\geq 0$ for all $x\in [0,r_0]$.  
\medskip

  \textbf{Step 6}.  ${\rm Ker}(L)$ has dimension 1.
\\We know that $\varphi_1\in {\rm Ker}(L)$, so ${\rm dim} ({\rm Ker}(L))\geq 1$. If ${\rm dim} ({\rm Ker}(L))> 1$, then there exist two linearly independent functions $w_1, w_2\in {\rm Ker}(L)$. Applying Gram-Schmidt orthonormalization procedure, we may assume that $\int w_1 \, w_2\, \dd \mm_{K,N}=0$.  By step 5, we may also assume that $w_1, w_2\geq 0$, forcing $w_1$ and $w_2$ to be concentrated on disjoint subsets. It follows that $w_1-w_2$ is an element of ${\rm Ker}(L)$ that changes sign, contradicting step 5.  We conclude that ${\rm dim} ({\rm Ker}(L))= 1$.
\medskip

 \textbf{Step 7}. $\varphi_1$ is positive, decreasing and Lipschitz on $[0,r_0]$.
\\So far we proved that the eigenfuction $\varphi_1$ is unique and non-negative up to multiplication by a non-zero constant, and it classically solves the differential equation \eqref{eq: eigen ode} on $(0,r_0)$. The fact that $\varphi_1'(\theta)< 0$ for all $\theta\in (0,r_0)$, if $\varphi_1$ is positive follows immediately: using \eqref{eq: eigen ode} we can infer that $\varphi_1$ satisfies
\begin{equation}\label{eq: identity phi'-phi}
\varphi_1'(\theta)h_{K,N}(\theta)=-\lambda\int_0^{\theta} \varphi_1(s)h_{K,N}(s)\dd s \qquad  \forall \theta\in (0,r_0),
\end{equation}
which implies $\varphi_1'(\theta)< 0$ for all $\theta\in (0,r_0)$, since $h_{K,N}>0$ in $(0,r_0]$. 

Using that
\[\lim_{\theta\downarrow 0}\frac{\mathfrak{s}_{\kappa}(\theta)}{\mathfrak{s}_{\kappa}'(\theta)}=0 \qquad \forall \kappa\in\mathbb{\theta},\]
we can also infer
\[
\lim_{\theta\downarrow 0} \frac{\int_0^{\theta} \varphi_1(s)h_{K,N}(s)\dd s}{h_{K,N}(\theta)}=0
\]
by applying L'H\^{o}pital's rule and recalling \eqref{eq: identity phi'-phi}, we have $\lim_{\theta\downarrow 0}\varphi_1'(\theta)=0$. In particular, $\varphi_1'$ is bounded in an interval $(0,\varepsilon_1)$ for a sufficiently small $\varepsilon_1>0$, from which we can easily conclude that $\varphi_1$ is Lipschitz in $[0,+\infty)$ thanks to the already proved continuity in $[0,+\infty)$ and Lipschitz regularity in $[\varepsilon,+\infty)$, for any $\varepsilon>0$.

The positivity of $h_{K,N}$ in the interior of the interval also implies, by a classical result in Sturm-Liouville theory, that $\varphi_1>0$ in $(0,r_0)$.
\end{proof}

\subsection{Localization}
In the proof of the main result of this note, we will use a localization argument based on the disintegration theorem. In order to make the paper as self-contained as possible, we briefly recall the results we will use (see \cite{cavalletti-mondino-inv,cavalletti-mondino-apde} for more details).

Let $(X,\sfd,\mm)$ be an essentially non-branching metric measure space satisfying the $\CD^{\star}(K,N)$ condition for some $K\in \R$ and $N\in (1,\infty)$. Fix a point $\bar{x}\in X$ and consider $u(\cdot):=\sfd(\bar{x}, \cdot)$, the distance function from $\bar{x}$. It is clear that $u$ is a $1$-Lipschitz function. Let
\begin{equation}\label{def:Gammau}
\Gamma_u:=\{(x,y)\in X\times X \colon u(x)-u(y)=\sfd(x,y)\},
\end{equation}
consider its transpose $\Gamma_u^{-1}:=\{(x,y)\in X\times X \colon (y,x)\in \Gamma_u\}$, and define
\begin{equation}\label{def:Ru}
R_u:=\Gamma_u\cup \Gamma_u^{-1}.
\end{equation}
Exploiting the essentially non-branching assumption, it is possible to prove that there exists a null set $\mathcal{N}\subset X$, $\mm(\mathcal{N})=0$, such that $R_u$ is an equivalence relation on
$$\tilde{X}:=X\setminus \mathcal{N}.$$
 The equivalence classes of such a relation are geodesics of $X$ (more precisely, geodesics emanating from $\bar{x}$). It follows that $\tilde{X}$ is partitioned by  a disjoint family of geodesics $\{X_q\}_{q\in Q}$, each of them isometric to an interval of $\R$. Here $Q$ is a suitable set of indices. In the smooth setting where $X$ is an $n$-dimensional Riemannian manifold $(M,g)$, we can choose $Q$ to be the unit sphere $S^{n-1}\subset T_{\bar{x}} M$, $q\in S^{n-1}$ a unit tangent vector, and $X_q$ the (maximally extended) geodesic obtained by exponentiating the vector $q$.

The partition $\{X_q\}_{q\in Q}$ of $\tilde{X}$ induces a decomposition of the reference measure $\mm$, via the disintegration theorem. Let us briefly recall the main elements of the construction and the associated properties.

The set of indices $Q$ can be identified with a measurable subset of $X$, intersecting each geodesic $X_q$ exactly ones (see \cite[Sect.\,3.1]{cavalletti-mondino-apde} for the details).
Moreover, one can define a measurable quotient map $\mathfrak{Q}:\tilde{X}\to Q$ induced by the partition  $\{X_q\}_{q\in Q}$:
\begin{equation}\label{eq:defQuotMap}
q=\mathfrak{Q}(x) \quad \Longleftrightarrow\quad x\in X_q.
\end{equation}
In the next statement,  $\mathcal{M}_+(X)$ denotes the space of non-negative Radon measures over $X$. For the proof we refer to \cite[Thm.\;3.4]{cavalletti-mondino-apde} and to the discussion in  \cite[Sec.\,3.B]{cavalletti-mondino-apde} after  \cite[Thm.\,5.1]{cavalletti-mondino-inv}

\begin{theorem}[\cite{cavalletti-mondino-inv,cavalletti-mondino-apde}]\label{th: loc}
Let $(X,\di,\mm)$ be an essentially non-branching metric measure space satisfying the $\mathsf{CD}^{\star}(K,N)$ condition, for some $K\in \R$ and $N\in (1,\infty)$.
Fix a point $\bar{x}\in X$ and let $u(\cdot):=\sfd(\bar{x}, \cdot)$ be the distance function from $\bar{x}$.

Then the reference measure $\mm$ admits the following disintegration formula:
\begin{equation}\label{eq:Dismm}
\mm=\int_Q \mm_{q}\,\mathfrak{q}(\dd q).
\end{equation}
Here, $\mathfrak{q}$ is a Borel probability measure over $Q\subset X$ such that $\mathfrak{Q}_{\sharp}\mm\ll \mathfrak{q}$ and the map
$$
Q\ni q \mapsto \mm_q \in \mathcal{M}_+(X)
$$
satisfies the following properties:
\begin{enumerate}
\item For any $\mm$-measurable set $B$, the map $q\mapsto \mm_q(B)$ is $\mathfrak{q}$-measurable.
\item \emph{Consistency}: For every  $\mm$-measurable set $B$ and any  $\mathfrak{q}$-measurable set $C$, the following disintegration formula holds:
$$
\mm(B\cap \mathfrak{Q}^{-1}(C))=\int_C\mm_q(B)\; \mathfrak{q}(\dd q).
$$
\item \emph{Strong consistency}: For $\mathfrak{q}$-a.e.\;$q\in Q$, the measure $\mm_q$ is concentrated on $X_q=\mathfrak{Q}^{-1}(q)$.
\item \emph{Rays are one-dimensional}: for $ \mathfrak{q}$-a.e.\;$q\in Q$, $X_q$ is either a point or a segment of a geodesic; in other words, $(X_q, \sfd_{| X_q\times X_q})$ is isometric to a real interval of the form $[0, r(q)]$, or $[0, r(q))$, or $(0, r(q)]$, or $(0, r(q))$, with $r(q)\in [0,+\infty]$.
\item For $\mathfrak{q}$-a.e.\;$q\in Q$, $\mm_{q}$ is a non-negative Radon measure satisfying $\mm_{q}=h_{q}\mathcal{H}^1\llcorner_{X_{q}}$ with $h_{q}$ a $\CD(K,N)$ density.
\end{enumerate}
\end{theorem}

\section{Proof of Cheng's eigenvalue theorem}\label{sec: proof Cheng}
This section is devoted to the proof of our main result, Theorem \ref{th: Cheng dir}. To do so, we take advantage of the localization technique that we have recalled in Theorem \ref{th: loc}, choosing $u(x):=\di(x_0,x)$ as guiding function.

A first crucial step in the proof of Cheng's bound is the analysis of the $1$-dimensional case, which is performed in the following theorem.
\begin{theorem}\label{th: 1dim}
Let $K\in \mathbb{R}$, $N\in(1,\infty)$, and let $r_0\in (0,D_{K,N})$. Let $h$ be a non-identically vanishing $\mathsf{CD}(K,N)$ density on an interval with left endpoint $0$ and right endpoint $b\ge r_0$; set $\mm:=h\mathcal{L}^1$. Denote by $\varphi_{K,N,r_0}$ the first positive Dirichlet eigenfunction for the interval $[0,r_0)$ in the model space $(I_{K,N},\di_{eu},\mm_{K,N})$. Then
\begin{equation}\label{eq: 1dimineq}
\int_{[0,\theta]}(\varphi_{K,N,r_0}')^2\,\dd \mm\le \lambda_{K,N,r_0}\int_{[0,\theta]}\varphi_{K,N,r_0}^2\,\dd \mm \qquad \forall \ \theta\in(0,r_0].
\end{equation}
Equality in \eqref{eq: 1dimineq} holds if and only if: 
 $\theta=r_0$ and there exists $c>0$ such that $h(x)=h_{K,N}(x)$, for every $x\in (0,r_0)$.

\end{theorem}
\begin{proof}
Since $h$ is a $\mathsf{CD}(K,N)$ density, then $h$ is positive and locally Lipschitz in $(0,r_0)$.

Recalling that $(\log{h})'\le (\log{h_{K,N}})'$ a.e.\;in $(0,r_0)$ (see \cite[Lemma A.9]{CaMi}),  and using \eqref{eq: eigen ode}, \eqref{eq: eigen-Decreas}, we get that 
\begin{equation}\label{eq: a.e. ineq}
\begin{aligned}
\lambda_{K,N,r_0}&\varphi^2_{K,N,r_0}h+\varphi_{K,N,r_0}(\varphi'_{K,N,r_0}h)'\\
&=\lambda_{K,N,r_0}\varphi^2_{K,N,r_0}h+\varphi_{K,N,r_0}h(\varphi_{K,N,r_0}''+(\log{h})'\varphi_{K,N,r_0}')\\
&\ge\lambda_{K,N,r_0}\varphi^2_{K,N,r_0}h+\varphi_{K,N,r_0}h(\varphi_{K,N,r_0}''+(\log{h_{K,N}})'\varphi_{K,N,r_0}')=0,
\end{aligned}
\end{equation}
 a.e.\;in  $(0,r_0)$.  Let $\theta\in (0,r_0]$ and $\varepsilon>0$. Since $h$ is locally semi-concave in $(0,r_0)$ and by the properties of Proposition \ref{prop: modeleigen}, we know that the functions $\varphi_{K,N,r_0}$ and $\varphi'_{K,N,r_0}h$ are absolutely continuous in $[\varepsilon,\theta-\varepsilon]$, so that 
\begin{align*}\int_{\varepsilon}^{\theta-\varepsilon}\left(\lambda_{K,N,r_0}\varphi^2_{K,N,r_0}-(\varphi_{K,N,r_0}')^2\right)\,\dd \mm+\Big[\varphi_{K,N,r_0}\varphi_{K,N,r_0}'h\Big]^{\theta-\varepsilon}_{\varepsilon}\\
=\int_{\varepsilon}^{\theta-\varepsilon}\left(\lambda_{K,N,r_0}\varphi^2_{K,N,r_0}h+\varphi_{K,N,r_0}(\varphi'_{K,N,r_0}h)'\right)\dd\mathcal{L}^1.
\end{align*}
We next take the limit as $\varepsilon\downarrow0$: using the dominated convergence theorem together with Proposition \ref{prop: modeleigen} in the left hand side, and the monotone convergence theorem together with \eqref{eq: a.e. ineq} in the right hand side, we infer that
\begin{align*}\int_{0}^{\theta}\left(\lambda_{K,N,r_0}\varphi^2_{K,N,r_0}-(\varphi_{K,N,r_0}')^2\right)\,\dd \mm+\lim_{x\uparrow \theta}\varphi_{K,N,r_0}(x)\varphi_{K,N,r_0}(x)'h(x)\\
=\int_{0}^{\theta}\left(\lambda_{K,N,r_0}\varphi^2_{K,N,r_0}h+\varphi_{K,N,r_0}(\varphi'_{K,N,r_0}h)'\right)\dd\mathcal{L}^1,
\end{align*}
where 
\[
\lim_{x\uparrow \theta}\varphi_{K,N,r_0}(x)\varphi_{K,N,r_0}(x)'h(x)=\begin{cases}
 \varphi_{K,N,r_0}(\theta)\varphi_{K,N,r_0}(\theta)'h(\theta)<0 \quad &\textrm{ if } \theta\in (0,r_0),   \\
 0 &\textrm{ if } \theta=r_0,
 \end{cases}
\]
as a consequence of the regularity and sign properties of $\varphi_{K,N,r_0}, \, \varphi_{K,N,r_0}'$ and $h$. 
Using again \eqref{eq: a.e. ineq}, we obtain the inequality \eqref{eq: 1dimineq}. Moreover, equality in \eqref{eq: 1dimineq} forces $\theta=r_0$ and $(\log{h})' \equiv (\log{h_{K,N}})'$ a.e. on $[0,r_0]$. Since $\log{h}$ is absolutely continuous in $(0,r_0)$, we can integrate the last identity and infer that $h(x)=ch_{K,N}(x)$, for every $x\in (0,r_0)$, for a suitable constant $c>0$.

\end{proof}

We can now prove our main result by taking advantage of the localization technique.
\begin{proof}[Proof of Theorem \ref{th: Cheng dir}]
It is sufficient to prove the Theorem assuming $r_0\in (0,D_{K,N})$, otherwise the inequality trivially holds. 
Let $\mathsf{d}_{x_0}:X\to [0,\infty)$ be the $1$-Lipschitz function $\di_{x_0}(x):=\di(x_0,x)$ and let $\varphi_{K,N,r_0}$ be the first positive Dirichlet eigenfunction (see Proposition \ref{prop: modeleigen}) of $[0,r_0)\subset I_{K,N}$, where $(I_{K,N},\di_{eu},\mm_{K,N})$ is the 1-dimensional model space (see Definition \ref{def: model space}). Since $\varphi_{K,N,r_0}$ is Lipschitz on $[0,+\infty)$ and $\varphi_{K,N,r_0}(\theta)=0$ for every $\theta\ge r_0$, it is immediate to see that $f_n=(\varphi_{K,N,r_0}\circ \di_{x_0}-\frac{1}{n})^{+}\in \Lip_{\textsf{c}}(B(x_0,r_0))$ and $f_n\to \varphi_{K,N,r_0}\circ \di_{x_0}$ in $W^{1,2}$ as $n\to +\infty$, so that by the definition of the local Sobolev space and the chain rule formula for local Sobolev functions (see \cite{Gigli} for all the details) we have 
\begin{equation}\label{eq:varphi'leqD}
\varphi_{K,N,r_0}\circ \di_{x_0}\in W^{1,2}_0(B(x_0,r_0)), \quad \textrm{with} \ |D(\varphi_{K,N,r_0}\circ \di_{x_0})|(x)\le |\varphi_{K,N,r_0}'|(\di_{x_0}(x)).
\end{equation}
 We consider localization in the space $(X,\di,\mm)$ with respect to $\di_{x_0}$, in the sense of Theorem \ref{th: loc}. Let $\tilde{r}(q):=\min\{r_0,r(q)\}$. By applying Theorem \ref{th: 1dim}, we obtain
\begin{equation}\label{eq: dismain Ray}
\begin{aligned}
\int_{B(x_0,r_0)} |D(\varphi_{K,N,r_0}\circ \di_{x_0})|^2\dd \mm& = \int_Q\left(\int_0^{\tilde{r}(q)} |D(\varphi_{K,N,r_0}\circ \di_{x_0})|^2\,\dd \mm_{q}\right)\mathfrak{q}(\dd q)\\
&\leq \int_Q\left(\int_0^{\tilde{r}(q)} (\varphi_{K,N,r_0}')^2\,\dd \mm_{q}\right)\mathfrak{q}(\dd q)\\
&\le \lambda_{K,N,r_0} \int_Q\left(\int_0^{\tilde{r}(q)}\varphi_{K,N,r_0}^2\,\dd \mm_{q}\right)\mathfrak{q}(\dd q)\\
&=\lambda_{K,N,r_0}\int_{B(x_0,r_0)}|\varphi_{K,N,r_0}\circ \di_{x_0}|^2\dd \mm\,.
\end{aligned}
\end{equation}
Since $\varphi_{K,N,r_0}\circ \di_{x_0}$ is a competitor in the definition \eqref{eq: def lambda_1^D} of $\lambda_1^D(B(x_0,r_0))$, the proof is concluded.
\end{proof}

We next establish rigidity in Theorem \ref{th: Cheng dir}. To this aim, we additionally assume that $(X,\di, \mm)$ is infinitesimally hilbertian, i.e.\;$(X,\di, \mm)$ is $\RCD^*(K,N)$.  Recall that $\RCD^*(K,N)$ spaces are essentially non-branching \cite{rajala-sturm}. 

In order to state the result, for $\kappa\in \R$ recall the definition \eqref{eq: sine func curv} of $\mathfrak{s}_\kappa(\cdot)$.
Moreover, we refer to \cite{Ket-JMPA} for the definition of $(K,N)$-cone; here, let us just recall that for $K=0$ it reduces to the standard metric cone, and that for $K>0$ it corresponds to a spherical suspension.

\begin{theorem}[Rigidity in Cheng's comparison]\label{th: Cheng dir-Rigid}
Let $(X,\di,\mm)$ be an $\RCD^{\star}(K,N)$ space, $K\in \R$ and $N\in (1,\infty)$.  Assume that equality holds in Cheng's comparison \eqref{eq:ChengBoundIntro} for some  $x_0\in X$ and $r_0>0$, i.e.
\begin{equation*}
\lambda_1^D(B(x_0,r_0))= \lambda_{K,N,r_0}.
\end{equation*} 
Then exactly one of the following cases holds:
\begin{enumerate}
\item \emph{$\partial B(x_0, r_0/2)$ contains only one point.} In this case $(X, \sfd)$ is isometric to $[0,{\rm Diam}(X)]$ ($[0,\infty)$ if $X$ is unbounded) with an isometry mapping $x_0$ to $0$ and the measure $\mm\llcorner{B(x_0, r_0)}$ to the measure $c\, \mathfrak{s}_{K/(N-1)}(x) \mathcal{L}^1(\dd x)$ for  $c:=\frac{\mm(B(x_0, r_0))}{\mm_{K,N}([0,r_0))}$.
\item  \emph{$\partial B(x_0, r_0/2)$ contains exactly two points.}  In this case $(X,\sfd)$ is a 1-dimensional Riemannian manifold, possibly with boundary. \\ More precisely, there exists a bijective local isometry (in the sense of distance-preserving maps) from $B(x_0, r_0)$ to $(-r_0, r_0)$ mapping $x_0$ to $0$ and the measure $\mm\llcorner{B(x_0, r_0)}$  to the measure $c\, \mathfrak{s}_{K/(N-1)}(|x|)\,  \mathcal{L}^1(\dd x)$ for $c:=\frac{\mm(B(x_0, r_0))}{2\mm_{K,N}([0,r_0))}$. Additionally, such a local isometry is an isometry when restricted to $\bar B(x_0, r_0/2)$.
\item \emph{$\partial B(x_0, r_0/2)$ contains at least  three points.} In this case there exists an $\RCD^*(N-2,N-1)$ space $Z$ such that if $Y$ is the $(K,N)$-cone over $Z$ with ``origin''  $O_Y$, then  $B(x_0, r_0)\subset X$  is locally isometric to $B(O_Y, r_0)\subset Y$. Additionally, such a local isometry is a measure-preserving bijection, which, when restricted to $\bar B(x_0, r_0/2)$, is an isometry.
\end{enumerate}
\end{theorem}

\begin{proof}
If equality holds in \eqref{eq:ChengBoundIntro} then necessarily also \eqref{eq: dismain Ray} holds with equality. It follows that
\begin{equation*}
\int_0^{\tilde{r}(q)} (\varphi_{K,N,r_0}')^2\,\dd \mm_{q}=  \lambda_{K,N,r_0} \int_0^{\tilde{r}(q)}\varphi_{K,N,r_0}^2\,\dd \mm_{q}, \quad \text{for } \mathfrak{q}\text{-a.e. }q\in Q.
\end{equation*}
From the rigidity statement in Theorem \ref{th: 1dim} we infer that, for $\mathfrak{q}$ a.e. $q\in Q$, either  $\mm_{q}$ is the null measure or
\begin{equation}\label{eq:hq=hKN}
\tilde{r}(q)=r_0 \ \ \text{ and }  \ \ \exists c_q>0 \ \textrm{ such that} \ \ {h}_q(t) = c_q\, h_{K,N}(t), \quad \text{ for every } t\in [0,\tilde{r}(q)].
\end{equation}
It follows that, for all $0<r<R\leq r_0$,
\begin{equation}\label{eq:VolumeConeCheng}
\begin{split}
\frac{\mm(B(x_0, r))}{\mm(B(x_0, R))}
&= \frac{\int_Q\left(\int_0^{r}  h_q(t)\,\mathcal{L}^1(\dd t)\right)\mathfrak{q}(\dd q)} {\int_Q\left(\int_0^{R}  h_q(t)\, \mathcal{L}^1(\dd t)\right)\mathfrak{q}(\dd q)}   \\
& = \frac{\int_Q\left(\int_0^{r}   c_q\, h_{K,N}(t)\,\mathcal{L}^1(\dd t)\right)\mathfrak{q}(\dd q)} {\int_Q\left(\int_0^{R}  c_q\, h_{K,N}(t)\, \mathcal{L}^1(\dd t)\right)\mathfrak{q}(\dd q)}  \\
& = \frac{\left(\int_0^{r} h_{K,N}(t)\,\mathcal{L}^1(\dd t) \right) \left( \int_Q   c_q\, \mathfrak{q}(\dd q) \right)} {\left(\int_0^{R} h_{K,N}(t)\,\mathcal{L}^1(\dd t) \right) \left( \int_Q   c_q\, \mathfrak{q}(\dd q) \right)}  \\
&= \frac{\mm_{K,N}([0, r))}{\mm_{K,N}([0, R))}
\end{split}
\end{equation}
where the first identity is obtained from the disintegration Theorem \ref{th: 1dim} combined with \eqref{eq:hq=hKN}, the second identity follows \eqref{eq:hq=hKN}, the third identity is a consequence of Fubini's Theorem, the last identity follows by simplifying the factor $\int_Q   c_q\, \mathfrak{q}(\dd q)\neq 0$ and recalling the Definition \ref{def: model space} of 1-dimensional model space. The rigidity statement  follows from the combination of \eqref{eq:VolumeConeCheng} and \cite[Thm.\;4.1]{DPG-GAFA}. 
\end{proof}

\begin{remark}
The rigidity in Cheng's original paper \cite{cheng-bound} states that if $(M,g)$ is a smooth $N$-dimensional Riemannian manifold with $\Ric_g\geq K g$ that contains a metric ball $B(x_0, r_0)$ achieving equality in Cheng's eigenvalue comparison,  then $B(x_0, r_0)$ is isometric to a metric ball in the $N$-dimensional space-form of constant sectional curvature $K/(N-1)$. Hence, Theorem \ref{th: Cheng dir-Rigid} seems a-priori weaker than the corresponding smooth version. Actually, this is not the case: indeed, in Differential Geometry, it is customary to say that two
Riemannian manifolds are isometric if the pull back of the metric tensors coincide.
However, it is easily seen that this implies only locally isometry as metric spaces.
Optimality of Theorem \ref{th: Cheng dir-Rigid}  (both in the smooth and non-smooth settings) can
be checked on the ball of radius 1 on the flat torus $\mathbb{T}^N$, $N\in \mathbb{N}$.
\end{remark}

We next prove stability in Cheng's comparison; namely, if equality is almost attained in \eqref{eq:ChengBoundIntro}, then the space is pmGH-close to a $(K,N)$-cone.  We denote by $\sfd_{\rm GH }$, $\sfd_{\rm pGH }$, $\sfd_{\rm mGH }$ the GH (Gromov-Hausdorff), pGH (pointed-Gromov-Hausdorff), mGH (measured-Gromov-Hausdorff) distances between compact metric spaces, complete metric spaces, compact metric measure spaces, respectively. We will also write $\sfd_{\rm eu}$ and  $g_{\rm eu}$ for the Euclidean distance and the Euclidean scalar product, respectively.

\begin{theorem}[Stability in Cheng's comparison]\label{th: Cheng dir-Stab}
For every $\delta>0,\;K\in \R,\; N\in (1,\infty),\; r_0>0$ there exists $\varepsilon=\varepsilon(\delta, K, N, r_0)>0$ with the following property.
Let $(X,\di,\mm)$ be an $\RCD^{\star}(K-\delta,N+\delta)$ space.  Assume that there exists $x_0\in X$ such that
\begin{equation*}
\lambda_1^D(B(x_0,r_0))\geq  \lambda_{K,N,r_0}-\delta.
\end{equation*} 

Then exactly one of the following cases holds:
\begin{enumerate}
\item  $\sfd_{\rm GH }\left((X, \sfd), ([0,{\rm Diam}(X)], \sfd_{\rm eu})\right)<\varepsilon$, or  $\sfd_{\rm pGH }\left((X, \sfd), ([0,\infty), \sfd_{\rm eu })\right)<\varepsilon$ in the case ${\rm diam}(X)=\infty$, with a quasi isometry mapping $x_0$ to $0$. Moreover, 
{\small
$$
\sfd_{\rm mGH }\left((\overline{B}(x_0, r_0), \sfd\llcorner{\overline{B}(x_0, r_0)}, \mm\llcorner{\overline{B}(x_0, r_0)}), ([0,r_0], \sfd_{\rm eu }, c\, \mathfrak{s}_{K/(N-1)}(x) \mathcal{L}^1(\dd x)  \right)<\varepsilon,$$} where $c:=\frac{\mm(B(x_0, r_0))}{\mm_{K,N}([0,r_0))}$.
\item There exists a 1-dimensional Riemannian manifold $(M,g_{\rm eu })$, possibly with boundary, such that $\sfd_{\rm pGH }\left((X, \sfd), (M, \sfd_{g_{\rm eu }})\right)<\varepsilon$. Moreover,  
{\footnotesize{$$\sfd_{\rm mGH }\left(\left(\overline{B}\left(x_0, \frac{r_0}{2}\right), \sfd\llcorner{\overline{B}\left(x_0, \frac{r_0}{2}\right)}, \mm\llcorner{\overline{B}\left(x_0, \frac{r_0}{2}\right)}\right), \left(\left[-\frac{r_0}{2},\frac{r_0}{2}\right], \sfd_{\rm eu }, c\, \mathfrak{s}_{K/(N-1)}(x) \mathcal{L}^1(\dd x)  \right) \right)<\varepsilon,$$}} 
where $c:=\frac{\mm(B(x_0, r_0))}{2\mm_{K,N}([0,r_0))}$.
\item There exists an $\RCD^*(N-2,N-1)$ space $Z$ such that if $Y$ is the $(K,N)$-cone over $Z$ with ``origin''  $O_Y$, then
$$
\sfd_{\rm mGH }\left(\overline{B}\left(x_0, \frac{r_0}{2}\right), \overline{B}\left(O_Y, \frac{r_0}{2}\right)  \right)<\varepsilon,
$$
where we wrote $\overline{B}\left(x_0, \frac{r_0}{2}\right)$ for $\left(\overline{B}\left(x_0, \frac{r_0}{2}\right), \sfd\llcorner{\overline{B}\left(x_0, \frac{r_0}{2}\right)}, \mm\llcorner{\overline{B}\left(x_0, \frac{r_0}{2}\right)} \right)$, and $\overline{B}\left(O_Y, \frac{r_0}{2}\right)$ for $\left(\overline{B}\left(O_Y, \frac{r_0}{2}\right), \sfd\llcorner{\overline{B}\left(O_Y, \frac{r_0}{2}\right)}, \mm\llcorner{\overline{B}\left(O_Y, \frac{r_0}{2}\right)} \right)$, for short.
\end{enumerate}
\end{theorem} 

\begin{proof}
Assume by contradiction there exists a sequence  of  $\RCD^{\star}(K-\delta_n,N+\delta_n)$ spaces  $(X_n,\di_n,\mm_n,\bar{x}_n)$, $n\in \N$, with $\delta_n=1/n$, such that
\begin{equation}\label{eq:landeltan}
\lambda_1^D(B(\bar{x}_n,r_0))\geq  \lambda_{K,N,r_0}-\delta_n,
\end{equation} 
and $(X_n,\di_n,\mm_n)$ does not satisfy any of the three cases in the thesis of the theorem, for some fixed $\varepsilon_0>0$; i.e.,  all of  them are $\varepsilon_0>0$ far from the rigid configurations. Then, Gromov's precompactness theorem and the stability of the $\RCD^{\star}$ condition imply that there exists a pointed $\RCD^{\star}(K,N)$ space $(X,\di,\mm,\bar{x})$ such that, up to subsequences, $(X_n,\di_n,\mm_n, \bar{x}_n)\to  (X,\di,\mm,\bar{x})$ in pmGH sense. It follows from \eqref{eq:landeltan} and \cite[Lemma 2.10]{AH18} that $\lambda_1^D(B(\bar{x},r_0))\geq  \lambda_{K,N,r_0}$, i.e. $B(\bar{x},r_0)$ achieves equality in Cheng's comparison \eqref{eq:ChengBoundIntro}. Theorem \ref{th: Cheng dir-Rigid} implies that $(X,\di,\mm,\bar{x})$ falls into one of the three rigid cases, giving a contradiction. 
\end{proof}




\begin{remark}\label{rem: expl value}
In most cases, $\lambda_{K,N,r_0}$ is not known explicitly. For $K=0$, the classical identity
\begin{equation}
	\lambda_{0,N,r_0}= \frac1{r_0^2} j_{\frac N2-1,1}
\end{equation}
holds for $N\in (1,\infty)$, where $j_{\nu,1}$ is the first zero of the Bessel function $J_\nu$. It is known that $j_{-1+N/2,1}^2< N (2+N/2)$. For $N=3$, the Bessel function simplifies, and $j_{1/2,1}=\pi$.

For $K\neq0$, $N\in (1,\infty)$ we can give upper bounds by using the methods of \cite{borisov-freitas}, which extend also to the case $N\notin \mathbb{N}$. Indeed their Lemma 3.1 only uses integration by parts in one dimension, and establishes that
\begin{equation}
	\frac1{r_0} \partial_{r_0} (r_0^2 \lambda(r_0))= \frac{
	\int_0^{r_0} d \theta \psi^2 h/\theta \partial_\theta (\theta^2 h^{-1/2} \partial_\theta^2 h^{1/2})}{\int_0^{r_0} \psi^2 h d\theta}\,
\end{equation}
with $h=h_{K,N}(\theta)$. Their Theorem 3.3 then becomes\footnote{In our language, \cite{borisov-freitas} considers the $|K|=N-1$ case.}
\begin{subequations}
\begin{align}
	\lambda_{K,N,r_0} \leqslant & -\frac16 N K  + \frac {j_{\frac N2-1,1}^2}{r_0^2}
	\quad	&(N<3);\\
 \label{exp.bound N=3}
		\lambda_{K,N,r_0}  = & -\frac12 K  + \frac{\pi^2}{r_0^2}
		\quad &(N=3);\\
  \label{exp.bound N>3}
		\lambda_{K,N,r_0} \leqslant & -\frac14(N-1)K +  \frac{j_{\frac N2-1,1}^2 }{r_0^2}& \nonumber\\
		& +\frac14(N-1)(N-3)\left(\frac1{\mathfrak{s}^2_{K/(N-1)}(r_0)} - \frac1{r_0^2}\right)  &  (N>3)\,.
\end{align}
\end{subequations}
The identity in the $N=3$ case is well-known and it can be established more directly.
\end{remark}

\subsection{Consequences}\label{sec: consequence}
In this section we collect some consequences of Theorem \ref{th: Cheng dir}. Before stating the results, recall that $\mathsf{RCD}(K,N)$ spaces are complete and proper, for $N<\infty$. Thus they are compact if and only if they have finite diameter. 

We start with an upper bound on the $j^{th}$ Neumann eigenvalues of the space, see \eqref{def: m-th Neum eigen}, generalizing the classical  \cite[Thm.\;2.1]{cheng-bound}  to the non-smooth setting.
\begin{theorem}\label{th: Cheng neum}
Let $(X,\di,\mm)$ be an essentially non-branching $\CD^{\star}(K,N)$ space with $\diam(X)=D<+\infty$, $K\in \R$ and $N\in (1,\infty)$. Then for every $j\in \N=\{1,2,...\}$ it holds
$$\lambda_j^N\le \lambda_{K,N,r_0}$$
with $r_0=\frac{D}{2j}.$
\end{theorem}

\begin{proof}[Proof of Theorem \ref{th: Cheng neum}]
Since $X$ is compact, we can find $y_0,y_1\in X$ such that $\diam(X)=\di(y_0,y_1)$. Let $\gamma$ be a geodesic joining $y_0$ and $y_1$ and let us consider $j+1$ points $\{x_i\}_{i=0}^j$ such that $x_i\in \gamma([0,1])$, $x_0=y_0$, $x_j=y_1$, and the balls $B(x_i,\frac{D}{2j})$ are pairwise disjoint.
Let $\varphi_i:=\varphi_{K,N,r_0}\circ \di_i$ where $\di_i(x):=\di(x_i,x)$, $i=0,\dots,j$. The estimate \eqref{eq: dismain Ray} implies the bound $\mathcal{R}(\varphi_i)\le \lambda_{K,N,r_0}$ on the Rayleigh quotient of $\varphi_i$, for every index $i$.

Now let us consider $j$ eigenfunctions $\psi_0,\dots,\psi_{j-1}$ relative to the first $j$ Neumann eigenvalues $\lambda_0^N,\dots,\lambda_{j-1}^N$, respectively. The map $L:\R^{j+1}\to \R^j$ whose $k$-th component is defined as
$$(a_0,\dots,a_{j})\mapsto \int_X \left(\sum_{i=0}^{j}a_i\varphi_i\right)\psi_k\,\dd \mm$$
is clearly linear, thus by Rouché–Capelli theorem we can find $\tilde{a}_0,\dots,\tilde{a}_{j}$ not all zero such that
\begin{equation}\label{eq: ort eigen proof}
\int_X \left(\sum_{i=0}^{j}\tilde{a}_i\varphi_i\right)\psi_k\,\dd \mm=0 \qquad \forall k=0,\dots, j-1.
\end{equation}

Since $\varphi_i$, $\varphi_k$ have disjoint supports for $i\neq k$, the function $\Phi:=\displaystyle \sum_{i=0}^{j}\tilde{a}_i\varphi_i$ is non-null, $\Phi\in W^{1,2}(X,\di,\mm)$ as a linear combination of $W^{1,2}$ functions, and it is orthogonal to the first $j$ eigenfunctions $\{\psi_k\}_{k=0}^j$ thanks to \eqref{eq: ort eigen proof}. By the very definition of $\lambda_j^N$ given in \eqref{def: m-th Neum eigen}, and using again that $\varphi_i$, $\varphi_k$ have disjoint supports for $i\neq k$,  we can conclude
$$\lambda_j^N\le \mathcal{R}(\Phi)=\frac{\sum_{i=0}^{j}\tilde{a}_i^2\mathcal{R}(\varphi_i)\int_X \varphi_i^2\,\dd \mm}{\sum_{i=0}^{j}\tilde{a}_i^2\int_X \varphi_i^2\,\dd \mm}\le \lambda_{K,N,r_0}$$
as desired.
\end{proof}

We also provide another application, related to a bound on the essential spectrum of the Laplacian in non-compact spaces. For smooth Riemannian manifolds, the result was proved in \cite[Thm. 3.1]{Donn81}, after \cite[Thm.\;4.2]{cheng-bound}. 
\begin{theorem}\label{th: essspect}
 Let $(X,\di,\mm)$ be a non-compact $\mathsf{RCD}(K,N)$ space, $K\le 0$ and $N\in [3, \infty)$.  Then the essential spectrum of $-\Delta$ intersects the interval $[0,-(N-1)K/4]$.
\end{theorem}
\begin{proof}
    We recall that $-\Delta$ is a self-adjoint operator in $L^2(X,\mm)$, so that using \cite[Prop.\;2.1]{Donn81} it is sufficient to prove that for any $\varepsilon>0$ there exists an infinite dimensional subspace $G_{\varepsilon}$ of the domain $D(\Delta)$ of the Laplacian, such that
  \begin{equation}\label{eq:Donnellycond}
    \int_{X}(-f\Delta f+(N-1)K/4f^2-\varepsilon f^2)\;\dd\mm<0,
   \end{equation}
    for every $f\in G_{\varepsilon}$.
    In order to prove \eqref{eq:Donnellycond},  fix $\varepsilon>0$ and notice that, using \eqref{exp.bound N=3},~\eqref{exp.bound N>3} and recalling that $(X,\di)$ is non-compact, we can find $\{x_i\}_{i=1}^{\infty}$, and $r>0$ sufficiently large such that $B(x_i,r)\cap B(x_j,r)=\emptyset$ for $i\neq j$ and
 \[
      \lambda_1^D(B(x_i,r)) \le \lambda_{K,N,r}< -\frac{1}{4}(N-1)K+\varepsilon/2 \qquad \forall i=1,2,\dots,
\]
where the first inequality follows by Theorem \ref{th: Cheng dir}. Set $\Omega:=\bigcup_{i=1}^\infty B(x_i,r)$. By the very definition of $\lambda_1^D(B(x_i,r))$, there exist $\{f_i\}_{i=1}^\infty\in W^{1,2}_0(\Omega)\subset W^{1,2}(X,\di,\mm)$ with  
$$\mathcal{R}(f_i)<- \frac{1}{4} (N-1) K+\frac{\varepsilon}{2}, \quad  \text{and} \quad  f_i\in \Lip_{\textsf{c}}(B(x_i,r)), \quad  \forall i=1,2,\dots.$$
 Using the density of $D(\Delta)$ in $W^{1,2}(X,\di,\mm)$ and reasoning as in the proof of \cite[Thm.\;7.8]{GMS},
one can find a sequence $\{\tilde{f}_i\}_{i=1}^\infty$ with $\tilde{f}_i\in D(\Delta)$, $\{\tilde{f}_i\}_{i=1}^k$ linearly independent and $\mathcal{R}(\tilde{f})<  -\frac{1}{4}(N-1)K+\varepsilon$ for every $\tilde{f}\in \mathrm{span}(\tilde{f}_1,\dots,\tilde{f}_k)$ and every $k\in \mathbb{N}$. Defining $G_{\varepsilon}:=\mathrm{span}(\{f_i\}_{i=1}^\infty)$ we have that \eqref{eq:Donnellycond} is satisfied after integrating by parts, and thus the proof is complete.

\end{proof}

\section{Physical applications}\label{sec: phys}

In the context of compactifications of higher-dimensional theories of gravity, eigenvalues of certain differential operators in the extra dimensions correspond to masses of particles observed in the four-dimensional spacetime. More generally, this correspondence holds for dimensional reduction to any number of dimensions, and these particles are known as \emph{Kaluza Klein (KK) particles}. Bounds on the eigenvalues of these operators directly translate to bounds on the masses of such particles.

Given a $D$-dimensional theory of gravity, a $d$-dimensional vacuum (with $d<D$) is a configuration in which the $D$-dimensional line element takes the form
\begin{equation}
	\text{ds}^2_D (x, y) \equiv \ee^{2 A (y)}\left( \text{ds}^2_{d, \Lambda} (x) +
	\text{ds}^2_n (y) \right) \,, \quad \;n \equiv D-d\;, \label{eq:metAns}
\end{equation}
where $\text{ds}^2_{d, \Lambda}$ is the metric on a maximally symmetric $d$-dimensional space-time (characterized by its cosmological constant $\Lambda$) and $A(y)$ is a function supported only on the extra dimensions, called the \emph{warp function}, or \emph{warping}. Fluctuations around these vacua are associated with particles in $d$ dimensions. In particular, the masses of the spin-2 Kaluza-Klein particles are given by the eigenvalues of the natural Laplacian on the extra dimensions associated with the line element $\text{ds}^2_n$ and weighted by $\ee^f$, with $f \equiv (D-2)A$ \cite{bachas-estes,csaki-erlich-hollowood-shirman,deluca-t-leaps}. Explicitly, this operator acts on a function $\psi$ as
\begin{equation}\label{eq:BELaplacian}
	\Delta_f (\psi) \equiv - \frac{1}{\sqrt{g}} \ee^{- f} \partial_m
	\left( \sqrt{g}  g^{m n} \ee^f \partial_n \psi \right) = \Delta \psi-\nabla f \cdot \dd \psi\,.
\end{equation}
Masses of spin-2 Kaluza-Klein fields are given by eigenvalues of the operator \eqref{eq:BELaplacian}. Recall that, if the space has no boundary, the eigenvalues of \eqref{eq:BELaplacian} coincide with the non-zero Neumann eigenvalues. The latter can be bounded using the results in this paper (in particular by Thm \ref{th: Cheng neum}), provided we can also bound the curvature in these spaces. Such curvature bounds can be obtained from the equations of motion of the theory. Indeed, for general warped compactifications, the classical equations of motion can be combined as \cite[Sec.~2.1]{deluca-deponti-mondino-t-entropy}
\begin{equation}
	\text{Ric}^{(2-d)}_{mn} = \Lambda g_{mn} + \tilde{T}_{mn}
\end{equation}
where
\begin{equation}
	\text{Ric}^{(q)}_{mn}\equiv	R_{mn} - \nabla_m \nabla_n f + \frac{1}{n-q} \nabla_m f \nabla_n f \;,
\end{equation}
and $\tilde{T}^{mn}$ is an appropriate combination of stress-energy tensors. When only classical sources are present in the background, with the exception of localized objects with negative tension (such as O-planes in string theory), $\tilde{T}^{mn}$ is non-negative \cite{deluca-t-leaps}, a condition named the Reduced Energy Condition (REC) in \cite{deluca-t-leaps}. In this case, we get the curvature bound
\begin{equation}\label{eq:curvEq}
	\text{Ric}^{(2-d)} \geqslant \Lambda\;.
\end{equation}
Smooth Riemannian manifolds satisfying \eqref{eq:curvEq} are $\RCD(K,N)$ spaces with $K = \Lambda$ and $N = 2-d <0$. This class also encompasses certain physical singularities, such as D-branes \cite{deluca-deponti-mondino-t,deluca-deponti-mondino-t-entropy}.
Theorems that bound the eigenvalues of \eqref{eq:BELaplacian} on $\RCD(K,N)$ spaces with $N<0$ thus directly translate into rigorous bounds on the spin-two Kaluza-Klein modes \cite[Sec.~6]{deluca-deponti-mondino-t-entropy}.

However, larger classes of eigenvalue bounds, such as those studied in this paper, are so far only available for $N>1$.
These two frameworks can be connected through an upper bound on the gradient of the warping.
Specifically, \eqref{eq:curvEq} can be equivalently rewritten as
\begin{equation}\label{eq:genN}
	\text{Ric}^{(N)}\geqslant \Lambda-\frac{N + d-2}{(D-2)(N-D+d)}\sigma^2\,,\qquad \sigma \equiv \sup{|\nabla f|}\,,\qquad \forall N> D-d.
\end{equation}

The case $N  = \infty$ was originally studied in \cite{deluca-t-leaps} for smooth configurations. Using the results in this work and the bound $\eqref{eq:genN}$, we can obtain more general results on the spectum of spin-2 KK particles.

Applying Thm.~\ref{th: Cheng neum}  we then obtain the following
\begin{proposition}\label{prop:kk}
	The mass of the $j$-th spin-2 Kaluza Klein fluctuation around a general $d$-dimensional vacum compactification \eqref{eq:metAns} of a $D$-dimensional gravitational theory that satifies the REC is bounded as
	\begin{equation}\label{eq:eq-prop}
		m^2_j \leqslant \lambda_{K, N, r_0}
	\end{equation}
	with
	\begin{equation}
		K = \Lambda-\frac{N + d-2}{(D-2)(N-D+d)}\sigma^2\;, \qquad \sigma \equiv \sup{|\nabla f|}\;,\qquad r_0 = \frac{\diam}{2j}\;,
	\end{equation}
	$\forall N > D-d$, where $\rm{diam}$ is the diameter of the internal space.
\end{proposition}
\begin{remark}
	In the smooth case, Prop.~\ref{prop:kk} follows directly from \eqref{th: Cheng neum} and \eqref{eq:genN}. To extend its validity to backgrounds with physical $\RCD$ singularities such as D-branes and M-branes \cite{deluca-deponti-mondino-t,deluca-deponti-mondino-t-entropy}, we first notice that some of these sources are at infinite distance \cite[Table 1]{deluca-deponti-mondino-t}, resulting in an infinite diameter that does not allow to apply the result. The sources at finite distance are D6, D7, and D8. For D7 and D8, the theorem can be directly applied, while the case of D6 branes requires more care since $\sigma^2 \to \infty$ approaching the source. Nonetheless, $\text{Ric}^{(N)}$ is still bounded from below on the D6; in fact, $\text{Ric}^{(N)} > 0$ in the local D6 solution (cf.~\cite[Sec.~2.2]{deluca-deponti-mondino-t}).
	We can then proceed as follows. Recall that locally a D6 source is completely determined by a harmonic function $H \sim (r/r_{D_6})^{-1}$, where $r_{D_6} = g_s l_s/(2\pi)$ and $g_s$ and $l_s$ respectively the string coupling and string length. Approaching the D6 and stopping at $r = x\,  r_{D_6}$, if $x$ is not too large the local D6 solution is a good approximation to the full solution for $r< x\, r_{D_6}$, and thus in this region we can replace the full solution with the local D6 solution, for which the curvature is non-negative. For solutions with D6 branes, a global weighted-Ricci curvature lower bound is then obtained as
	\begin{equation}
		K = \Lambda-\frac{N + d-2}{(D-2)(N-D+d)}\hat{\sigma}^2\;, \qquad \hat \sigma \equiv \sup_{\hat{M}}{|\nabla f|\,,}
	\end{equation}
	where $\hat M$ denotes the internal manifold with all the D6 branes excised at distance $r = x\, r_{D_6}$. At the joining locus $|\nabla f| = \frac{\pi}{x\, g_s l_s}$, which can be made small for $x \gg (g_s l_s)^{-1}$. But the further we go from the D6 pole the less we can trust \ a priori the general local solution: this forbids us to write a strong eigenvalue bound that applies to \emph{any} solution with D6. Nonetheless, such a bound can be easily computed on a case-by-case basis by directly computing the weighted-Ricci curvature lower bound $K$ on the given solution.

\end{remark}

\begin{remark}
While in Thm.~\ref{th: Cheng neum} the quantities $K$ and $N$ are independent, for our physical application in Prop.~\ref{prop:kk}, $K$ depends on $N$, which is free to take any value $> D-d$. A natural question is then to determine the behavior of the bound as a function of $N$. Considering the generic case $N>3$, we can further bound the right-hand side of \eqref{eq:eq-prop} using \eqref{exp.bound N>3}. Keeping all other variables fixed, this quantity diverges to $+\infty$ both as $N \to \infty$ and $N \to D-d$, implying an optimal value of $N$, which can be determined numerically as a function of the other quantities. Even for general $N$, these new bounds improve the ones obtained in \cite[Thm.~2]{deluca-t-leaps} in the case $N = \infty$; indeed, the latter bounds had an exponential dependence on the geometric dimension.

\end{remark}

\bibliography{at}
\bibliographystyle{siam}

\end{document}